\documentclass[11pt]{article}

\usepackage{amsmath,amsthm,amsfonts,amssymb,mathrsfs,epsfig,bm,mathrsfs,a4}
\usepackage[usenames]{color}
\usepackage[colorlinks=true]{hyperref}
\usepackage{graphicx}
\parskip \medskipamount
\newtheorem{theorem}{Theorem}%[section]
\newtheorem{lemma}{Lemma}
\newtheorem{corollary}{Corollary}
\newtheorem{proposition}{Proposition}

\theoremstyle{definition}
\newtheorem{remark}{Remark}

\def\N{\mathbb{N}}

\def\P{\mathbb{P}}
\def\E{\mathbb{E}}

\renewcommand{\phi}{\varphi}
\renewcommand{\epsilon}{\varepsilon}

\newcommand{\1}{{\text{\Large $\mathfrak 1$}}}

\newcommand{\var}{\operatorname{var}}

\renewcommand{\limsup}{\varlimsup}
\renewcommand{\liminf}{\varliminf}

\definecolor{mygray}{gray}{0.9}
\definecolor{deeppink}{RGB}{255,20,147}
\definecolor{mygreen}{rgb}{0.05, 0.576, 0.03}
\definecolor{myred}{rgb}{0.768, 0.09, 0.09}

\long\def\symbolfootnote[#1]#2{\begingroup
\def\thefootnote{\fnsymbol{footnote}}\footnote[#1]{#2}\endgroup}
\usepackage{MnSymbol}

\usepackage{authblk}
\allowdisplaybreaks

\title{Probabilities of large values for sums of i.i.d.\ non-negative random variables with regular tail of index $-1$}
\date{\today}
\author{Matthias Birkner\footnote{Johannes-Gutenberg-Universit\"{a}t Mainz,
Institut f\"{u}r Mathematik\\
Staudingerweg 9, 55099 Mainz, Germany\\Email: \texttt{birkner@mathematik.uni-mainz.de}} {} and Linglong Yuan\footnote{University of Liverpool,
Department of Mathematical Sciences\\
Peach Street, L69 7ZL, Liverpool, UK\\Email: \texttt{yuanlinglongcn@gmail.com}}
\footnote{Xi'an Jiaotong-Liverpool University,
Department of Mathematical Sciences\\
Ren'ai Road, 215111, Suzhou, China}
}
\begin{document}
\maketitle

\begin{abstract}
    Let $\xi_1, \xi_2, \dots$ be i.i.d.\ non-negative random variables whose tail varies regularly with index $-1$, let $S_n$ be the sum and $M_n$ the largest of the first $n$ values. We clarify for which sequences $x_n\to\infty$ we have $\P(S_n \ge x_n) \sim \P(M_n \ge x_n)$ as $n\to\infty$. Outside this regime, the typical size of $S_n$ conditioned on exceeding $x_n$ is not completely determined by the largest summand and we provide an appropriate correction term which involves the integrated tail of $\xi_1$.\\
    
    
\end{abstract}
\section{Motivation and main results}
Let $\xi_1,\xi_2,\ldots$ be independent and identically distributed (i.i.d.) real-valued random variables, put
$$S_n:=\sum_{i=1}^n\xi_i, \quad M_n:=\max\{\xi_1,\xi_2,\ldots,\xi_n\}.$$
There is a considerable body of literature analysing 
\begin{equation}\label{snx}\overline F_n(x):=\P(S_n\geq x), \text{ as } n\to\infty, x\to\infty,\end{equation}
usually assuming that the right tail of $\xi_1$,
\begin{equation}\label{righttail}\overline F(t):=\P(\xi_1\geq t),\end{equation}
and/or the left tail, $F(-t):=\P(\xi_1\leq -t)$, are/is regularly varying in $t$ as $t\to\infty$.  See \cite{B00} and \cite{B03} and references therein.  Additionally, many variations of \eqref{snx} have been considered, such as letting the number $n$ be random, see for instance \cite{GM77} and  \cite{DFK10}. 

A natural class of models in mathematical population genetics, which can describe the evolution
of a population of fixed size $n$ and incorporate substantial variation of the individuals' reproductive
success in every generation (in an elegantly ``tuneable'' manner), describes reproduction as a two-step
procedure as follows: For $i=1,2,\dots,n$ let individual $i$ produce $\xi_i$ ($\ge 0$) potential offspring or ``juveniles'',
where the $\xi_i$ are i.i.d.\ as above. Then, sample $n$ times without replacement from this pool of juveniles to form the next adult generation. See, for example, \cite{S03}, \cite{BLS18} and \cite{BY21}.
When analysing such a model, it is important to understand \eqref{snx} and the contribution to $S_n$ from different individuals when $S_n$ is large. In this application, the $\xi_i$'s take nonnegative integers.
Motivated by an application in population genetics \cite{BY21}, in this paper we consider the case that $\xi_i$'s are i.i.d.\ and take nonnegative real values such that the tail \eqref{righttail} is regularly varying with index $-1$ as $x\to\infty$. In this case, the population models are related to the famous Bolthausen-Sznitman coalescent, see, e.g., the case $a=1$ in \cite[Section~3]{S03}.

When \eqref{righttail} varies regularly with an index different from $-1$, the questions we address here have already been answered, see, e.g.\ \cite{N79}, \cite{Na82}, \cite{S03}. However, no-one we asked seemed aware of those results we needed for the case of index $-1$. The $\xi_i$'s we consider belong to the family of subexponential random variables. See (\cite{FKZ11}, section 3.2) for necessary and sufficient conditions for a real-valued random variable to be subexponential. 

For subexponential random variables, many researchers have investigated the following problem: Find  sequences $(x_n)_{n\geq 1}$ such that 
\begin{equation}\label{sxi}\lim_{n\to\infty}\sup_{x\geq x_n}\left|\frac{\overline F_n(x)}{n\overline F(x)}-1\right|=0.\end{equation}
Note that $\P(M_n\geq x) = 1 - (1-\overline F(x))^n \sim n\overline F(x)$ uniformly in $x \ge x_n$ if $\overline F(x_n)=o(n^{-1})$. Thus \eqref{sxi} has the interpretation that the (sufficiently far out) tail of the sum of i.i.d.\ random variables is determined by the tail of the largest summand. Interested readers can find results on this topic in \cite{MN98, N79, V94}. Denisov {\it et al} \cite{DDS08} provided a unifying approach to find a sufficient condition on $(x_n)_{n\geq 1}$ for \eqref{sxi} to hold. Moreover a local version of \eqref{sxi} was proved in \cite{DDS08} (see $\Delta$-subexponentiality, a notion introduced in \cite{SSD03}). In this paper we focus on $\overline F_n(x)$ only. 

For subexponential random variables, conditional on $S_n\geq x$ for $x$ sufficiently large and satisfying \eqref{sxi}, Armend\'ariz and Loulakis \cite{AL11} proved that the $n-1$ smallest random variables behave like $n-1$ i.i.d.\ random variables of the original distribution. The results of this paper also cover the cases satisfying the local version of \eqref{sxi}. In the following, we will study the $n-1$ smallest variables from a different perspective, namely their contributions to $S_n$. 

The novelty of this paper is that we provide a necessary and sufficient condition for \eqref{sxi} to hold in the case that $\xi_i$'s are i.i.d.\ and nonnegative, with the right tail regularly varying of index $-1.$ Furthermore, in this case, we characterise the behaviour of $S_n$ and $M_n$ when the threshold $x_n$ goes to infinity but \eqref{sxi} does not hold.

\begin{theorem}\label{main}Let $\xi_1,\xi_2,\ldots$ be i.i.d.\ nonnegative random variables with the right tail regularly varying of index $-1$: 
\begin{equation}\label{xic}
\overline F(x)=x^{-1}\ell(x),\quad  x\to\infty,
\end{equation} 
where $\ell(x)$ is a slowly varying function. Define 
$$G(x):=\int_0^x\overline F(t)dt, \quad x\geq 0.$$
Then 
\begin{enumerate}
\item[(a).] The convergence \eqref{sxi} holds if and only if \begin{equation}\label{xnginf}\lim_{n\to\infty}\frac{x_n}{nG(x_n)}=\infty.\end{equation}
\item[(b).] Let $(x_n)_{n\geq 1}$ and $(x_n')_{n\geq 1}$ be two positive sequences such that 
\begin{equation}\label{xngfini}x_n\leq x_n',\quad \liminf_{n\to\infty}\frac{x_n}{nG(x_n)}>0,\quad \limsup_{n\to\infty}\frac{x_n'}{nG(x_n')}<\infty.\end{equation}
Then for any positive sequence $(a_n)_{n\geq 1}$ such that $\lim_{n\to\infty}a_n\in(0,\infty)$, 
\begin{equation}\label{xx'}\lim_{n\to\infty}\sup_{x_n\leq x\leq x_n'}\left|\frac{\overline F_n\left(nG(x)+a_nx\right)}{n\overline F(a_nx)}-1\right|=0.\end{equation}
\item[(c).] If $\limsup_{n\to\infty}\frac{x_n}{nG(x_n)}=0$, then 
$$\lim_{n\to\infty} \overline F_n(x_n) = 1.$$
\end{enumerate}
 \end{theorem}
\begin{remark} It will follow from the proof of Theorem~\ref{main} that the function $\frac{x}{G(x)}$ is strictly increasing for $x$ larger than a certain number. Thus, the above theorem exhibits the behaviour of $\overline F_n(x)$ for different magnitudes of $x$. 
\end{remark}

\begin{remark}
1.\ Observe that \eqref{xic} can be compatible with $\E[\xi_1]<\infty$, for example when $\ell(x)=1/\log^2(x)$. In this case, $G(x) \to \E[\xi_1]$ as $x \to \infty$ and 
\eqref{xnginf} simply means that $x_n$ grows faster than $n$.

\noindent
2.\ When $\ell(x) \to c \in (0,\infty)$ as $x\to\infty$, we have $G(x) \sim c \log(x)$ and the conditions in \eqref{xngfini} mean that $x_n, x_n'$ are of order $n \log n$.
\end{remark}

Next result shows the contributions of different terms to $S_n$. 

\begin{theorem}\label{contri}
\label{|} Under the assumptions of Theorem~\ref{main}, let $(x_n)_{n\geq 1}$ be any positive sequence satisfying 
\begin{equation}\label{xn}\liminf_{n\to\infty}\frac{x_n}{nG(x_n)}>0\end{equation}
and let $(a_n)_{n\geq 1}$ be a positive sequence with $\lim_{n\to\infty}a_n=:a\in(0,\infty)$. 

Among all the $i$'s for which $\xi_i$ is equal to $M_n$, we select one uniformly. Denote the chosen index by $i_*$ and let $\Xi$ be the set of indices of the remaining $n-1$ terms. Let $\Xi_k$ be a (possibly random) subset of $\Xi$ of size $k$ for $1\leq k\leq n-1$, which is drawn independently of the values of $\xi_i, i \neq i_*$. Let $S_{n,k}$ be the sum of random variables whose indices are in $\Xi_k$.  Then the distribution of $S_{n,k}$ is independent of the particular set $\Xi_k$ we choose and for any $\epsilon>0$ as $n\to\infty,$
\begin{equation}\label{condm}\max_{1\leq k\leq n-1}\P\left(|S_{n,k}-kG(x_n)|\geq \epsilon x_n\,\, \,|\,\,M_n\geq a_n x_n\right)=o(1),\end{equation}
and 
\begin{equation}\label{conds}\max_{1\leq k\leq n-1}\P\left(|S_{n,k}-kG(x_n)|\geq \epsilon x_n \,\, \,|\,\, S_{n}-nG(x_n)\geq a_nx_n\right)=o(1).\end{equation}
\end{theorem}
We will see in Proposition \ref{naga82} below that as $n\to\infty,$
$$\P\left(M_n\geq a_n x_n\right)\sim \P\left(S_{n}-nG(x_n)\geq a_nx_n, \,\,\, M_n\geq a_n x_n\right)\sim \P\left(S_{n}-nG(x_n)\geq a_nx_n\right).$$
Thus, the two events $\{M_n\geq a_n x_n\}$ and $\{S_{n}-nG(x_n)\geq a_nx_n\}$ are almost identical in the above sense. Furthermore, Theorem \ref{contri} shows that,  conditional on $M_n\geq a_nx_n$ or $S_n-nG(x_n)\geq a_nx_n$, the deviation of $S_n-nG(x_n)$ from $0$ is provided by the maximum variable. In fact, we see from Theorem~\ref{contri} that 
the sum of any subset of size $k$ of the other $n-1$ variables typically deviates from $kG(x_n)$ on a scale smaller than $x_n$.

The phenomenon that the maximum variable contributes substantially to $S_n$ while the others fluctuate around a certain level, when $S_n$ is large enough, can also be observed when $\overline F(x)$ is regularly varying with index strictly less than $-1$ and thus the $\xi_i$'s have a finite mean. This can be seen in the proof of the main result in \cite{Na82}, although the author did not formulate results like \eqref{condm} and \eqref{conds}. If the index lies in $(-1,0)$, this phenomenon does not occur, see (\cite{S03}, Lemma 22) for a study in the case of nonnegative random variables. 

The paper is organised as follows: In Section~\ref{sect:asymptS_n}, we prove the key result Propsition~\ref{naga82}. In Section~\ref{sect:asymptS_k}, we investigate how $S_k$ fluctuates around a certain level for $1\leq k\leq n$ and prove Theorem \ref{main}. Finally, Theorem~\ref{contri} is proved in Section~\ref{sect:ProofThm2}.

\section{Proofs}
\subsection{Asymptotic behaviour of $S_n$ and $M_n$}
\label{sect:asymptS_n}

\begin{lemma}\label{slow}
Let $g:[0,\infty)\mapsto [0,\infty)$ be a function tending to infinity at infinity. For any two slowly varying functions $h, h':[0,\infty)\mapsto [0,\infty)$, there exists $f:[0,\infty)\mapsto [0,\infty)$ such that $\lim_{x\to\infty} f(x)=\infty$, $\lim_{x\to\infty} f(x)/x=0$ and $\lim_{x\to\infty} f(x)g(x)/x=\infty$, and
\begin{equation}\label{hf}\lim_{x\to\infty}\frac{h(x)}{h(f(x))}=1, \quad \lim_{x\to\infty}\frac{h'(x)}{h'(f(x))}=1.\end{equation}
\end{lemma}
\begin{proof}
%Define $g^*:[0,\infty)\mapsto [0,\infty)$ by $g^*(x)=\inf_{y\geq x} g(y).$ By assumption, $g^*(x)$ increases to infinity at infinity. 

Since $h$ and $h'$ are slowly varying, there exist sequences of integers $(k_n)_{n\geq 1}, (k_n')_{n\geq 1}$ strictly increasing to infinity such that 
\begin{equation}\label{nx} \left|\frac{h(x)}{h(x/n)}-1\right|\leq 2^{-n}, \,\, \forall n\in \N_+, x\geq k_n;\quad \left|\frac{h'(x)}{h'(x/n)}-1\right|\leq 2^{-n}, \,\, \forall n\in \N_+, x\geq k_n'\end{equation}
and (by possibly increasing $k_n$ and $k_n'$ suitably) we may assume that also
$$\lim_{n\to\infty}\frac{g(k_n)}{n}=\infty,\quad \lim_{n\to\infty}\frac{g(k_n')}{n}=\infty,\quad \lim_{n\to\infty}\frac{k_n}{n}=\infty,\quad \lim_{n\to\infty}\frac{k_n'}{n}=\infty.$$

Define $f(x)=x$ for $0 \le x <  \max\{k_1,k_1'\}$ and 
$$f(x)=x/n \quad \text{for } x\in \left[\max\{k_n,k_n'\}, \max\{k_{n+1},k_{n+1}'\}\right), \; n \in \N_+.$$
%; \quad f(x)=x,\quad \forall x\in\left[0,\max\{k_1,k_1'\}\right),$$
By construction, $\lim_{x\to\infty}f(x)=\infty$, $\lim_{x\to\infty}f(x)/x=0$ and $\lim_{x\to\infty} f(x)g(x)/x=\infty$. %Since $g(x)\geq g^*(x)$ for any $x\geq 0$, we have $\lim_{x\to\infty} f(x)g(x)/x=\infty$. 
Thus \eqref{nx} implies \eqref{hf}, completing the proof. 
\end{proof}

%{\tt 29.3.2021: We do not think that %this is necessary (and there are %counterexamples)}
%Since $\overline F$ is regularly varying %with index $-1$, we get 
%$$\lim_{x\to\infty}G(x)=\infty.$$
%Moreover using (\cite{BGT89}, Proposition 1.5.9a, page 26),  $G(x)$ is a slowly varying function as $x\to\infty$ and 
%\begin{equation}\label{gl}\lim_{x\to\infty}\frac{G(x)}{\ell(x)}=\infty.\end{equation}
%Let $C', C$ be two positive numbers such that 
%$$C' x^{-1}\leq \overline F(x)\leq C x^{-1}, \quad \forall %\,x>1.$$
\begin{proposition}\label{naga82}
Let $(x_n)_{n\geq 1}$, $(a_n)_{n\geq 1}$ and $a = \lim_n a_n$ be as defined in Theorem \ref{contri}. Then as $n\to\infty,$ 
\begin{align}\label{sim}\P(S_n-nG(x_n) \geq a_nx_n)&\sim \P(S_n-nG(x_n) \geq a_nx_n, \,\,\, M_n\geq a_nx_n)&\nonumber\\
&\sim \P(M_n\geq a_nx_n)\sim a^{-1}n\overline F(x_n)=o(1).&\end{align}
\end{proposition}
\begin{proof}
First of all, let us observe what condition \eqref{xn} implies for $(x_n)_{n\geq 1}$. Note that $G(x)\sim \overline F(0+)x$ as $x\to 0$. Therefore 
$$\lim_{n\to\infty}\frac{x_n'}{nG(x_n')}=0$$
for any sequence $x_n' \to 0$. 
In particular, \eqref{xn} enforces
$$\liminf_{n\to\infty}x_n>0.$$
This together with \eqref{xn} implies that in fact
\begin{equation}\label{xinf}\lim_{n\to\infty}x_n=\infty.\end{equation}
Since $\overline F$ is regularly varying with index $-1$, using Karamata's theorem (\cite{BGT89}, Theorem 1.5.11, page 28), we get 
\begin{equation}\label{gxf}\lim_{x\to\infty}\frac{G(x)}{x \overline F(x)}=\lim_{x\to\infty}\frac{G(x)}{\ell(x)}=\infty.\end{equation}
Together with \eqref{xn}, we obtain  
\begin{equation}\label{nv}\lim_{n\to\infty}n\overline F(x_n)=0,\end{equation}
which is the right-most claim in (the second line of) \eqref{sim}.

For the maximum value $M_n$, it is straightforward to see that 
\begin{equation}\label{mninf}\P(M_n\geq a_nx_n)=1-(1-\P(\xi_1\geq a_nx_n))^n\sim n\overline F(ax_n)\sim a^{-1}n\overline F(x_n)=o(1), \quad \text{ as } n\to\infty.\end{equation}
Here we used \eqref{nv} and the uniform convergence theorem for regularly varying functions (\cite{BGT89}, Theorem 1.2.1, page 6), as $a_n\to a\in(0,\infty)$.
It remains to show that as $n\to\infty$
\begin{equation}\label{remain}\P(S_n-nG(x_n) \geq a_nx_n)\sim \P(S_n-nG(x_n) \geq a_nx_n, \,\,\, M_n\geq a_nx_n)\sim a^{-1}n\overline F(x_n).\end{equation}

Since $\ell(x)$ is slowly varying as $x\to\infty$,  $G(x)$ is also slowly varying as $x\to\infty$, see (\cite{BGT89}, Proposition 1.5.9b, page 27). We can choose a positive sequence $(y_n)_{n\geq 1}$ such that as $n\to\infty,$
\begin{align}&y_n\to\infty, \quad y_n=o(x_n),&\label{yinf}\\&ny_n^2\overline F(x_n)=o(1),&\label{nyv}\\
&\ell\left(\frac{x_n}{y_n}\right)\sim \ell(x_n), \text{ or equivalently,  } \overline F\left(\frac{x_n}{y_n}\right)\sim y_n\overline F(x_n),&\label{ll}\\
&G\left(\frac{x_n}{y_n}\right)\sim G(x_n).&\label{gslow}
\end{align}
Such a sequence exists by Lemma \ref{slow}: Define $m_{x_n}:=\sup\{j: x_j=x_n, j\geq 1\}$ for $n\geq 1$. Note that ($n \le$)\,$m_{x_n}$  is finite for all $n$ due to \eqref{xinf}. By \eqref{nv}, there exists a function $g:[0,\infty)\mapsto [0,\infty)$ tending to infinity at infinity such that $g(x_n)=\frac{1}{\sqrt{m_{x_n}\overline F(x_n)}}, \forall n\geq 1.$ Let $h=l, h'=G.$ Applying Lemma \ref{slow}, there exists $f$ such that 
$$\lim_{x\to\infty} f(x)=\infty, \,\lim_{x\to\infty} f(x)/x=0,\, \lim_{x\to\infty} f(x)g(x)/x=\infty,\, \lim_{x\to\infty}\frac{h(x)}{h(f(x))}=1, \, \lim_{x\to\infty}\frac{h'(x)}{h'(f(x))}=1.$$
Then it suffices to set $y_n=x_n/f(x_n), n\geq 1$. 

Define \begin{equation}\label{defxi'}\xi_i':=\xi_i\1_{\xi_i< \frac{x_n}{y_n}}, \quad i\geq 1.\end{equation} Since $\P(\xi_1'\geq \frac{x_n}{y_n})=0,$ we use integration by parts to obtain that as $n\to\infty,$
\begin{align}\label{xi'}\E[\xi_1']=\int_0^\infty\P(\xi_1'\geq y) \, dy&=\int_{0}^{\frac{x_n}{y_n}}\Big(\P(\xi_1\geq y)-\P(\xi_1\geq \frac{x_n}{y_n})\Big) \, dy&\nonumber\\
&=G\left(\frac{x_n}{y_n}\right)-\frac{x_n}{y_n}\overline F\left(\frac{x_n}{y_n}\right)&\nonumber\\
&=G\left(\frac{x_n}{y_n}\right)-\ell\left(\frac{x_n}{y_n}\right)\sim G(x_n)-\ell(x_n)\sim G(x_n).&\end{align}
For the two equivalences in the last line, we used \eqref{ll}, \eqref{gslow}, and \eqref{xinf}, \eqref{gxf}, respectively.

Define \begin{equation}\label{defsk'}S_{k}':=\sum_{i=1}^{k}\xi_i', \quad k\geq 1.\end{equation} Then by \eqref{xi'}, we have $\E[S_{n}']\sim nG(x_n).$ Moreover 
\begin{equation}\label{xis'}\var(\xi_1')\leq \E[\xi_1'^2]=\int_0^\infty 2y\P(\xi_1'\geq y)\,dy\leq \int_0^{\frac{x_n}{y_n}} 2y\P(\xi_1\geq y)\,dy\sim \frac{2x_n^2\overline F(x_n)}{y_n}.\end{equation}
For the last equivalence, we used again Karamata's theorem and \eqref{ll}. Let $C>0$ satisfy that 
\begin{equation}\label{C}\var(\xi_1')\leq C\frac{x_n^2\overline F(x_n)}{y_n},\quad \forall n\geq 1.\end{equation}

%Then  we have 
%\begin{equation}\label{xis'}\var(\xi_1')\leq 2C\frac{x_n}{y_n}; \quad \var(S'_{k})=k\var(\xi_1')\leq 2Ck\frac{x_n}{y_n}, \,\,\forall k=1,2,\ldots.\end{equation}

Let $0<s<1.$ Consider the following decomposition:
\begin{align}\label{1234}
&\quad \P(S_n-nG(x_n) \geq a_nx_n)&\nonumber\\
&=n\P\left(S_n-nG(x_n) \geq a_nx_n, \,\,\,\xi_n\geq sa_nx_n,\,\,\, \max_{1\leq i\leq n-1}\xi_i< \frac{x_n}{y_n}\right)&\nonumber\\
&\quad \quad +n\P\left(S_n-nG(x_n) \geq a_nx_n, \,\,\, \frac{x_n}{y_n} \leq \xi_n< sa_nx_n, \,\,\, \max_{1\leq i\leq n-1}\xi_i< \frac{x_n}{y_n}\right)&\nonumber\\
&\quad\quad\quad \quad +\P\left(S_n-nG(x_n) \geq a_nx_n, \,\,\, \max_{1\leq i\leq n}\xi_i< \frac{x_n}{y_n}\right)&\nonumber\\
&\quad\quad\quad \quad \quad\quad+\P\left(S_n-nG(x_n) \geq a_nx_n, \,\,\, \exists \,1\leq i<j\leq n \text{ such that } \,\xi_i\geq \frac{x_n}{y_n},\xi_j\geq \frac{x_n}{y_n}\right)&\nonumber\\
&=I_1+I_2+I_3+I_4&
\end{align}
where $I_i$ is the $i$-th term on the right side for $1\leq i\leq 4$. Let us look at $I_4$ first. Note that 
\begin{align}\label{i2}0\leq I_4&\leq {n\choose 2}\P\left(\xi_1\geq  \frac{x_n}{y_n}, \,\,\, \xi_2\geq  \frac{x_n}{y_n}\right)\leq n^2\P\left(\xi_1\geq  \frac{x_n}{y_n}\right)\P\left(\xi_2\geq  \frac{x_n}{y_n}\right)&\nonumber\\
&= n^2\overline F^2\left(\frac{x_n}{y_n}\right)\sim n^2y_n^2\overline F^2(x_n)=o(n\overline F(x_n)), \text{ as }n\to\infty.&\end{align}
For the last two equivalences, we used \eqref{ll} and \eqref{nyv}, respectively. 

Now we consider $I_2.$ Since $\E[S'_{n-1}]\sim nG(x_n)$, using \eqref{xn} and $a_n\to a>0$, we have  
$$\left\{S'_{n-1}\geq nG(x_n)+(1-s)a_nx_n\right\} \mbox{\Large$\subset$}\left\{S'_{n-1}-\E[S'_{n-1}]\geq \frac{1-s}{2}a_nx_n\right\} \quad \text{ for } n  \text{ large enough.}$$
Then we obtain that
\begin{align}\label{i3}
0\leq I_2&\leq n\P\left(\xi_n\geq \frac{x_n}{y_n},  \,\,\, S_{n-1}-nG(x_n) \geq (1-s)a_nx_n, \,\,\, \max_{1\leq i\leq n-1}\xi_i< \frac{x_n}{y_n}\right)&\nonumber\\
&=n\P\left(\xi_n\geq \frac{x_n}{y_n},\,\,\,  S'_{n-1}-nG(x_n) \geq (1-s)a_nx_n, \,\,\,\max_{1\leq i\leq n-1}\xi_i< \frac{x_n}{y_n}\right)&\nonumber\\
&=n\P\left(\xi_n\geq \frac{x_n}{y_n}\right)\P\left(S'_{n-1}-nG(x_n) \geq (1-s)a_nx_n,\,\,\, \max_{1\leq i\leq n-1}\xi_i< \frac{x_n}{y_n}\right)&\nonumber\\
&\leq n\P\left(\xi_n\geq \frac{x_n}{y_n}\right)\P(S'_{n-1}\geq nG(x_n)+(1-s)a_nx_n)&\nonumber\\
&\leq n\overline F\left(\frac{x_n}{y_n}\right)\P\left(S'_{n-1}-\E[S'_{n-1}]\geq \frac{1-s}{2}a_nx_n\right)&\nonumber\\
&\leq n\overline F\left(\frac{x_n}{y_n}\right)\frac{\var(S'_{n-1})}{\left(\frac{1-s}{2}a_nx_n\right)^2}&\nonumber\\
&\leq C\left(\frac{2}{(1-s)a_n}\right)^2\frac{n^2\overline F(x_n)}{y_n}\overline F\left(\frac{x_n}{y_n}\right)&\nonumber\\
&\sim C\left(\frac{2}{(1-s)a}\right)^2(n\overline F(x_n))^2=o(n\overline F(x_n)),\quad \text{ as }n\to\infty.&
\end{align}
For the last inequality we applied \eqref{C}. We used \eqref{ll} and \eqref{nv} respectively for the first and second equivalence in the last line. 

We apply the same approach to $I_3$:
\begin{equation}\label{i4}0\leq I_3\leq \P\left(S'_n-nG(x_n) \geq \frac{a_nx_n}{2}\right)\leq \P\left(S'_{n}-\E[S'_n]\geq \frac{a_nx_n}{4}\right)\leq \frac{16C}{a_n^2}\frac{n\overline F(x_n)}{y_n}=o(n\overline F(x_n)),\end{equation}
where we used $\lim_{n\to\infty}y_n=\infty$ in the right-most equality.

Now we proceed to estimate $I_1$, recall that it remains to show that $I_1 \sim \frac{1}{a} n \overline{F}(x_n)$.
Let $0<\delta<1-s.$
\begin{align}\label{1123}
I_1&=n\P\left(S_n-nG(x_n) \geq a_nx_n, \,\,\, \xi_n\geq sa_nx_n,\,\,\, \max_{1\leq i\leq n-1}\xi_i< \frac{x_n}{y_n}\right)&\nonumber\\
&= n\P\left(\xi_n+S'_{n-1}-nG(x_n) \geq a_nx_n, \,\,\,\xi_n\geq sa_nx_n, \,\,\,\max_{1\leq i\leq n-1}\xi_i< \frac{x_n}{y_n}\right)&\nonumber\\
&=n\P\left(\xi_n+S'_{n-1}-nG(x_n) \geq a_nx_n, \,\,\,\xi_n\geq sa_nx_n\right)&\nonumber\\
&\quad\quad\quad  -n\P\left(\xi_n+S'_{n-1}-nG(x_n) \geq a_nx_n,\,\,\, \xi_n\geq sa_nx_n, \,\,\,\max_{1\leq i\leq n-1}\xi_i\geq \frac{x_n}{y_n}\right)&\nonumber\\
&= n\P\left(\xi_n+S'_{n-1}-nG(x_n) \geq a_nx_n,\,\,\, \xi_n\geq sa_nx_n, \,\,\,\left|\frac{S'_{n-1}-\E[S'_{n-1}]}{a_nx_n}\right|< \delta\right)&\nonumber\\
&\quad\quad\quad +n\P\left(S'_{n-1}-nG(x_n) \geq a_nx_n,\,\,\, \xi_n\geq sa_nx_n, \,\,\,\left|\frac{S'_{n-1}-\E[S'_{n-1}]}{a_nx_n}\right|\geq \delta\right)&\nonumber\\
&\quad\quad\quad\quad\quad  -n\P\left(\xi_n+S'_{n-1}-nG(x_n) \geq a_nx_n, \,\,\,\xi_n\geq sa_nx_n, \,\,\,\max_{1\leq i\leq n-1}\xi_i\geq \frac{x_n}{y_n}\right)&\nonumber\\
&=I_{11}+I_{12}+I_{13}&
\end{align}
where $I_{1i}$ is the $i$-th term on the right side for $i=1,2,3.$ 
Note that 
\begin{align*}
0\leq I_{12}&\leq n\P\left(\xi_n\geq sa_nx_n, \,\,\,\left|\frac{S'_{n-1}-\E[S'_{n-1}]}{a_nx_n}\right|\geq \delta\right)&\\
&=n\P(\xi_n\geq sa_nx_n) \P\left(\left|\frac{S'_{n-1}-\E[S'_{n-1}]}{a_nx_n}\right|\geq \delta\right)&\\
&\leq n\overline F(sa_nx_n)\frac{\var(S'_{n-1})}{a_n^2\delta^2x_n^2}\leq n\overline F(sa_nx_n)\frac{Cn\overline F(x_n)}{a_n^2\delta^2y_n}\sim \frac{Cn^2\overline F^2(x_n)}{sa_n^3\delta^2y_n}=o(n\overline F(x_n)).&
\end{align*}
In the last inequality we used \eqref{C}. For the two equivalences, we applied the uniform convergence theorem of regularly varying functions and  \eqref{yinf}, \eqref{nv}. 

Note that 
$$\P\left(\max_{1\leq i\leq n-1}\xi_i\geq \frac{x_n}{y_n}\right)=1-\left(1-\P\left(\xi_i\geq \frac{x_n}{y_n}\right)\right)^{n-1}\sim n\overline F\left(\frac{x_n}{y_n}\right)\sim ny_n\overline F(x_n).$$
Thus we can estimate $|I_{13}|$ as follows: 
\begin{align}\label{i13}
0\leq |I_{13}|&\leq n\P\left(\xi_n\geq sa_nx_n, \max_{1\leq i\leq n-1}\xi_i\geq \frac{x_n}{y_n}\right)&\nonumber\\
&=n\overline F(sa_nx_n)\P\left(\max_{1\leq i\leq n-1}\xi_i\geq \frac{x_n}{y_n}\right)\sim \frac{1}{sa}n^2y_n\overline F^2(x_n)=o(n\overline F(x_n)).
\end{align}
For the last term we applied \eqref{nyv} and \eqref{yinf}.

Finally, we consider the main term $I_{11}$. For any $\epsilon>0$ such that $\delta(1+\epsilon)<1-s$, using \eqref{xn} and $a_n\to a>0$, it follows that $\left|\frac{S'_{n-1}-\E[S'_{n-1}]}{a_nx_n}\right|< \delta$ implies 
\begin{equation}\label{<s<}nG(x_n)-\delta(1+\epsilon)a_nx_n<S'_{n-1}< nG(x_n)+\delta(1+\epsilon)a_nx_n, \quad  \text{for } n \text{ large enough}.\end{equation}
Then we use the second inequality to obtain 
$$\left\{\xi_n+S'_{n-1}-nG(x_n) \geq a_nx_n, \,\,\,\xi_n\geq sa_nx_n, \,\,\,\left|\frac{S'_{n-1}-\E[S'_{n-1}]}{a_nx_n}\right|< \delta\right\}\mbox{\Large$\subset$}\{\xi_n\geq (1-\delta(1+\epsilon))a_nx_n\}.$$
The above display leads to
\begin{align}\label{i11<}
0\leq I_{11}\leq n\P(\xi_n\geq (1-\delta(1+\epsilon))a_nx_n)\sim (1-\delta(1+\epsilon))^{-1}a^{-1}n\overline F(x_n).\end{align}
By the first inequality of \eqref{<s<}, we have
\begin{align*}
&\left\{\xi_n\geq (1+\delta(1+\epsilon))a_nx_n, \left|\frac{S'_{n-1}-\E[S'_{n-1}]}{a_nx_n}\right|< \delta\right\}&\\
&\quad\quad\quad \mbox{\Large$\subset$} \left\{\xi_n+S'_{n-1}-nG(x_n) \geq a_nx_n, \,\,\,\xi_n\geq sa_nx_n, \left|\frac{S'_{n-1}-\E[S'_{n-1}]}{a_nx_n}\right|< \delta\right\}, \quad  \text{for } n \text{ large enough}.\end{align*}
Then 
\begin{align}\label{i11>}
I_{11}&\geq n\P\left(\xi_n\geq (1+\delta(1+\epsilon))a_nx_n, \left|\frac{S'_{n-1}-\E[S'_{n-1}]}{a_nx_n}\right|< \delta\right)&\nonumber\\
&=n\P(\xi_n\geq (1+\delta(1+\epsilon))a_nx_n)\P\left(\left|\frac{S'_{n-1}-\E[S'_{n-1}]}{a_nx_n}\right|< \delta\right)&\nonumber\\
&\geq n\overline F\left((1+\delta(1+\epsilon))a_nx_n\right)\left(1-\frac{\var(S_{n-1}')}{(\delta a_n x_n)^2}\right)&\nonumber\\
&\geq n\overline F\left((1+\delta(1+\epsilon))a_nx_n\right)\left(1-\frac{Cn\overline F(x_n)}{\delta^2 a_n^2y_n}\right)\sim (1+\delta(1+\epsilon))^{-1}a^{-1}n\overline F(x_n).&
\end{align}
To achieve the last two steps we used similar arguments as in the estimation of $I_{12}$. 
 
Taking into account (\ref{1234})-\eqref{i13} and \eqref{i11<}-(\ref{i11>}) and choosing $s$ close enough to $1$ (so $\delta$ close to $0$),  we obtain that 
\begin{align}\label{2equi}
  & \P\left(S_n-nG(x_n) \geq a_nx_n\right)  \nonumber\\
  & \qquad \sim a^{-1}n\overline F(x_n)
    \sim n\P\left(S_n-nG(x_n) \geq a_nx_n, \,\,\, M_n=\xi_n\geq a_nx_n,\,\,\, \max_{1\leq i\leq n-1}\xi_i< \frac{x_n}{y_n}\right).\end{align}
Moreover the third term is bounded above by 
$$\P(S_n-nG(x_n) \geq a_nx_n, \,\,\, M_n\geq a_nx_n).$$
The above term is  again bounded above by 
$$\P(M_n\geq a_nx_n).$$
Then by \eqref{2equi} and \eqref{mninf}, we get 
$$\P(S_n-nG(x_n) \geq a_nx_n, \,\,\, M_n\geq a_nx_n)\sim a^{-1}n\overline F(x_n).$$
This completes the proof of \eqref{remain} hence also of Proposition~\ref{naga82}.
\end{proof}
\begin{remark}\label{better}
The two equivalences in \eqref{2equi} entail that, as $n\to\infty$
$$\P\left(\parbox{0.39\textwidth}{Only one term is larger than $a_nx_n,$\\the rest are all smaller than  $x_n/y_n$}
  % \text{Only one term larger than } a_nx_n, \text{ the rest are all smaller than }  \frac{x_n}{y_n}
  \: \Big| \: S_n-nG(x_n) \geq a_nx_n\right)\to 1.$$
Applying the second equivalence in \eqref{2equi} and using the subset relationship, we have  
$$n\P\left(M_n=\xi_n\geq a_nx_n, \max_{1\leq i\leq n-1}\xi_i< \frac{x_n}{y_n}\right)\sim a^{-1}n\overline F(x_n)\sim \P(M_n \geq a_nx_n), \text{ as } n\to\infty.$$
In other words, as $n\to\infty$
$$\P\left(\parbox{0.39\textwidth}{Only one term is larger than $a_nx_n,$\\the rest are all smaller than  $x_n/y_n$}
  % \text{Only one term larger than } a_nx_n, \text{ the rest are all smaller than }  \frac{x_n}{y_n}
  \: \Big| \: \: M_n\geq a_nx_n\right)\to 1.$$
\end{remark}
%It is not very clear from Proposition \ref{naga82} and the above remark to see how big the contributions to $S_n$ from different terms, especially between the maximum value and the others, conditional on $\{S_n-nG(x_n) \geq a_nx_n\}$. We discuss this problem in the next section and prove Theorem \ref{contri}. 

\subsection{Asymptotic behaviour of $S_k$}
\label{sect:asymptS_k}

In this section, we investigate the behaviour of $S_k-kG(x_n)$ for $1\leq k\leq n$. In particular, we will see that
in the regime we consider, uniformly in $k$, $|S_k-kG(x_n)|$ will typically be much smaller than $a_nx_n$.
After that we will be able to prove Theorem \ref{main}. 
 
\begin{corollary}\label{firstmax}
Let $(x_n)_{n\geq 1}$, $(a_n)_{n\geq 1}, a$ be as defined in Theorem \ref{contri}. Then as $n\to\infty$
$$\max_{1\leq k\leq n}\P(S_k-kG(x_n)\geq a_nx_n)=O(n\overline F(x_n)).$$
%and 
%$$\max_{1\leq k\leq n}\P(S_k-kG(x_n)\leq -a_nx_n)=o(nx_n^{-1}). $$
\end{corollary}
\begin{proof}
Define $\xi_i':=\xi_i\1_{\xi_i\leq x_n}$ for $i\geq 1$ and $S_k':=\sum_{i=1}^k\xi_i', k\geq 1$.  Similarly as in \eqref{xi'} and \eqref{xis'}, we have 
$$\E[\xi_1']\sim G(x_n); \quad \E[S_k']\sim kG(x_n), \forall k=1,2,\ldots, \text{ as }n\to\infty,$$
and there exists $C'>0$ such that for any $n\geq 1,$
$$\var(\xi_1')\leq C'x_n^2\overline F(x_n); \quad \var(S'_{k})=k\var(\xi_1')\leq C'kx_n^2\overline F(x_n), \,\,\forall k=1,2,\ldots.$$
Note that the following inequality holds
\begin{align*}
 \P\Big(\left|S_k'-kG(x_n)\right|\geq a_nx_n\Big)\leq \P\Big(\left|S_k'-\E[S_k']\right|\geq a_nx_n-\left|\E[S_k']-kG(x_n)\right|\Big).
\end{align*}
Note also that
$$\frac{|\E[S_k']-kG(x_n)|}{x_n}= k\frac{|\E[\xi_1']-G(x_n)|}{x_n}=\frac{kG(x_n)}{x_n}\frac{|\E[\xi_1']-G(x_n)|}{G(x_n)}\leq \frac{nG(x_n)}{x_n}\frac{|\E[\xi_1']-G(x_n)|}{G(x_n)}.$$
The first factor in the last term is bounded, due to \eqref{xn}. The second factor converges to $0$, as $n\to\infty$, since $\E[\xi_1']\sim G(x_n)$ and $G(x_n)\to\infty$.
Then for $n$ large enough,
$$ \P\left(|S_k'-kG(x_n)|\geq a_nx_n\right)\leq  \P\left(|S_k'-\E[S_k']|\geq a_nx_n/2\right)\leq 4\frac{\var(S_k')}{a_n^2x_n^2}\leq \frac{4C'k\overline F(x_n)}{a_n^2},\quad \forall 1\leq k\leq n.$$
Therefore
\begin{equation}\label{||}
 \max_{1\leq k\leq n}\P\left(|S_k'-kG(x_n)|\geq a_nx_n\right)=O(n\overline F(x_n)), \text{ as }n\to\infty.
\end{equation}
Now we consider the statement of the corollary. Note that
\begin{align}\label{2term}
&\quad\P(S_k-kG(x_n)\geq a_nx_n)&\nonumber\\
&=\P\left(S_k-kG(x_n)\geq a_nx_n, \,\,\,\max_{1\leq i\leq k}\xi_i\leq x_n\right)+\P\left(S_k-kG(x_n)\geq a_nx_n,\,\,\, \max_{1\leq i\leq k}\xi_i> x_n\right)&\nonumber\\
&\leq \P\left(|S_k'-kG(x_n)|\geq a_nx_n\right)+\P\left(\max_{1\leq i\leq k}\xi_i> x_n\right).&\end{align}
Note that 
$$\P\left(\max_{1\leq i\leq k}\xi_i> x_n\right)=1-(1-\P(\xi_1> x_n))^k\leq k\overline F(x_n).$$
Together with \eqref{||}, the statement is proved. 

%For the second statement,  we have 
%$$\P(S_k-kG(x_n)\leq -a_nx_n)\leq \P(S_k'-kG(x_n)\geq -a_nx_n)\leq $$

\end{proof}

\begin{corollary}\label{twoscor}
Let $(x_n)_{n\geq 1}$, $(a_n)_{n\geq 1}, a$ be as defined in Theorem \ref{contri}.  Let $(y_n)_{n\geq 1}$ be any positive sequence as specified in the proof of Proposition \ref{naga82}. 
Then with the constant $C>0$ given in \eqref{C},  for $n$ large enough, we have 
$$\P(S_k-kG(x_n)\leq -a_nx_n)\leq \frac{4Ck\overline F(x_n)}{a^2_ny_n}, \quad \forall 1\leq k\leq n.$$
\end{corollary}
%\begin{remark}
%The condition on $(y_n)_{n\geq 1}$ is weaker compared to \eqref{yn} as here we do not require $ny_n^2=o(x_n)$. 
%\end{remark}
\begin{proof}
Let $(\xi_i')_{1\leq i\leq n}$ and $(S_k')_{k\geq 1}$ be the terms defined respectively in \eqref{defxi'} and \eqref{defsk'}. Then 
$$\P(S_k-kG(x_n)\leq -a_nx_n)\leq \P(S_k'-kG(x_n)\leq -a_nx_n).$$
Recall \eqref{xi'} and \eqref{C}. For any fixed $\epsilon>0$ and large enough $n$, we have 
$$(1-\epsilon)kG(x_n)\leq \E[S_k']\leq (1+\epsilon)kG(x_n).$$ 
By \eqref{xn} and $a_n\to a>0,$
we have 
$$\left\{S_k'-kG(x_n)\leq -a_nx_n\right\}\mbox{\Large$\subset$}\left\{S_k'-\E[S_k']\leq -a_nx_n/2\right\}\mbox{\Large$\subset$}\left\{\left|S_k'-\E[S_k']\right| \geq a_nx_n/2\right\}, \quad \text{ for } n \text{ large enough}.$$
The above three displays lead to 
$$\P\left(S_k-kG(x_n)\leq -a_nx_n\right)\leq \P\left(|S_k'-\E[S_k']| \geq a_nx_n/2\right)\leq \frac{4}{a_n^2}\frac{\var(S_k')}{x_n^2}\leq\frac{4Ck\overline F(x_n)}{a_n^2y_n}.$$
\end{proof}

\begin{proof}[Proof of Theorem \ref{main}]
  Assertion $(b)$ follows from Proposition \ref{naga82}. If we prove $(c)$, then by the monotonicity of $\frac{x}{G(x)}$ (to be proved below), Assertion $(a)$ holds true, using Proposition \ref{naga82}.

  So it remains to prove $(c)$. Consider the function $\frac{x}{G(x)}, x\geq 0$. Note that 
$$\lim_{x\to 0}\frac{x}{G(x)}=\frac{1}{\overline F(0+)}.$$
Moreover $\overline F(x)$ is continuous to the left and has a limit to the right at every point $x>0$, and decreasing from $1$ to $0.$ Then we have 
$$\frac{x}{G(x)}-\frac{1}{\overline F(0+)}=\frac{x}{\int_0^x\overline F(t)dt}-\frac{1}{\overline F(0+)}=\int_0^x\frac{\int_0^t\overline F(s)ds-t\overline F(t)}{(\int_0^t\overline F(s)ds)^2}dt=\int_0^x\frac{G(t)-t\overline F(t)}{G(t)^2}dt.$$
Using \eqref{gxf}, there exists a number $t_0>0$ such that $G(t)>t\overline F(t)$ for any $t\geq t_0.$ Therefore $\frac{x}{G(x)}$ is strictly increasing for $x\geq t_0.$

Let $(x_n)_{n\geq 1}$ be a positive sequence such that 
$$0<a:=\liminf_{n\to\infty}\frac{x_n}{nG(x_n)}\leq b:= \limsup_{n\to\infty}\frac{x_n}{nG(x_n)}<\infty.$$
Then by Corollary \ref{firstmax} and Corollary \ref{twoscor} and \eqref{nv}, \eqref{yinf},  for any $\epsilon\in (0,1),$
\begin{equation}\label{xbig}\lim_{n\to\infty}\P\left(\frac{1-\epsilon}{b}\leq \frac{S_n}{x_n}\leq \frac{1+\epsilon}{a}\right)=1.\end{equation}

Let $(y_n)_{n\geq 1}$ be another positive sequence such that 
$$\limsup_{n\to\infty}\frac{y_n}{nG(y_n)}=0.$$
By the monotonicity of $\frac{x}{G(x)}$, we have $y_n\leq x_n$ for $n$ large enough. Moreover due to $G(x)$ being slowly varying, it must be that $y_n=o(x_n)$ as $n\to\infty.$ Then \eqref{xbig}  implies that 
$$\lim_{n\to\infty}\P(S_n\geq y_n)=1.$$
Thus, Assertion $(c)$ is proved. 
\end{proof}

\subsection{Contributions to $S_n$ from different terms when $S_n$ is large}
\label{sect:ProofThm2}

\begin{proof}[Proof of Assertion \eqref{condm} in Theorem \ref{contri}]
Without loss of generality, say $\xi_n=\xi_*$. Then the other variables $\xi_1,\xi_2,\ldots,\xi_{n-1}$ are exchangeable.
Therefore, the distribution of $S_{n,k}$ does not depend on the particular set $\Xi_k.$ Now we prove the statement in \eqref{condm}. By definition,
\begin{align*}
\P\big(|S_{n,k}-kG(x_n)| \geq \epsilon x_n \,\,\, \big| \,\, M_n\geq a_n x_n\big)&=\frac{\P(|S_{n,k}-kG(x_n)|\geq \epsilon x_n, \,\,\, M_n\geq a_n x_n)}{\P(M_n\geq a_n x_n)}.&\end{align*}
For the numerator, we have
$$\big\{|S_{n,k}-kG(x_n)|\geq \epsilon x_n, \,\,\,M_n\geq a_n x_n\big\}=\bigcup_{i=1}^n\big\{|S_{n,k}-kG(x_n)|\geq \epsilon x_n, \,\,\,\xi_i=\xi_*\geq a_n x_n\big\}.$$
The above two displays entail that 
\begin{align*}
\P\big(|S_{n,k}-kG(x_n)|\geq \epsilon x_n \,\,\, \big| \,\, M_n\geq a_n x_n\big)&\leq n\frac{\P\big(|S_{n,k}-kG(x_n)|\geq \epsilon x_n,\,\,\, \xi_n=\xi_*\geq a_n x_n\big)}{\P(M_n\geq a_n x_n)}&\\
&=n\frac{\P\big(|S_k-kG(x_n)|\geq \epsilon x_n,\,\,\, \xi_n=\xi_*, \,\,\,\xi_n\geq a_n x_n\big)}{\P(M_n\geq a_n x_n)}&\\
&\leq n\frac{\P\big(|S_k-kG(x_n)|\geq \epsilon x_n,\,\,\, \xi_n\geq a_n x_n\big)}{\P(M_n\geq a_n x_n)}&\\
&=n\frac{\P\big(|S_k-kG(x_n)|\geq \epsilon x_n\big)\P(\xi_n\geq a_n x_n)}{\P(M_n\geq a_n x_n)}&\\
&\sim n\frac{\P\big(|S_k-kG(x_n)|\geq \epsilon x_n\big)a^{-1}\overline F(x_n)}{a^{-1}n\overline F(x_n)}.&
\end{align*}
Then we use Corollary \ref{firstmax} and \ref{twoscor}, equations \eqref{nv} and \eqref{yinf} to conclude that \eqref{condm} holds. 
\end{proof}

\begin{proof}[Proof of Assertion \eqref{conds} in Theorem \ref{contri}]
Note that 
\begin{align*}
\P\big( & |S_{n,k}-kG(x_n)|\geq \epsilon x_n \,\,\, \big| \,\, S_{n}-nG(x_n)\geq a_nx_n\big)&\\
& = \frac{\P\big(|S_{n,k}-kG(x_n)|\geq \epsilon x_n,  \,\,\, S_{n}-nG(x_n)\geq a_nx_n\big)}{\P(S_{n}-nG(x_n)\geq a_nx_n)}&\\
& = \frac{\P\big(|S_{n,k}-kG(x_n)|\geq \epsilon x_n,  \,\,\,S_{n}-nG(x_n)\geq a_nx_n,\,\,\, M_n\geq a_n x_n\big)}{\P(S_{n}-nG(x_n)\geq a_nx_n)}&\\
&\quad\quad +\frac{\P\big(|S_{n,k}-kG(x_n)|\geq \epsilon x_n, \,\,\, S_{n}-nG(x_n)\geq a_nx_n,\,\,\, M_n< a_n x_n\big)}{\P(S_{n}-nG(x_n)\geq a_nx_n)}&\\
& \leq \frac{\P\big(|S_{n,k}-kG(x_n)|\geq \epsilon x_n, \,\,\,M_n\geq a_n x_n\big)}{\P(S_{n}-nG(x_n)\geq a_nx_n)}+\frac{\P(S_{n}-nG(x_n)\geq a_nx_n, \,\,\,M_n< a_n x_n)}{\P(S_{n}-nG(x_n)\geq a_nx_n)}&\\
& =\frac{\P(M_n\geq a_n x_n)}{\P(S_{n}-nG(x_n)\geq a_nx_n)}\P\big(S_{n,k}-kG(x_n)\geq \epsilon x_n \,\,\, \big| \,\, M_n\geq a_n x_n\big)&\\
&\quad\quad +\P(M_n<a_n x_n\,\,\, | \,\, S_{n}-nG(x_n)\geq a_nx_n).&
\end{align*}
Then we use \eqref{sim} and \eqref{condm} to conclude that \eqref{conds} holds. 
\end{proof}

\section*{Acknowledgements}
The authors thank Sergey Foss for helpful comments and for pointers to the relevant literature. The authors also thank Dmitry Korshunov and Priscilla Greenwood for their friendly comments and references.  Linglong Yuan acknowledges support of the National Natural Science Foundation of China (Youth Programme, grant 11801458). Matthias Birkner and Linglong Yuan were in part supported by Deutsche Forschungsgemeinschaft through DFG priority programme 1950 Probabilistic Structures in Evolution, grant BI 1058/2-2.

\end{document}